\documentclass[12pt]{amsart}

\usepackage{graphicx, url}
\usepackage{verbatim}
\usepackage{amsmath,amsthm,amsfonts}
\usepackage{amssymb}
\usepackage{times}
\usepackage{enumitem}
\usepackage[all]{xy}
\usepackage{esvect}

\newtheorem{theorem}{Theorem}
\newtheorem{proposition}[theorem]{Proposition}

\newtheorem{lemma}[theorem]{Lemma}

\theoremstyle{definition}

\newtheorem{def-theorem}[theorem]{Definition-Theorem}
\newtheorem{remark}[theorem]{Remark}

\newtheorem{question}[theorem]{Question}

\newcommand{\be}{\begin{equation}}
\newcommand{\ee}{\end{equation}}
\newcommand{\bea}{\begin{eqnarray}}
\newcommand{\eea}{\end{eqnarray}}
\newcommand{\beas}{\begin{eqnarray*}}
\newcommand{\eeas}{\end{eqnarray*}}
\newcommand{\ba}{\begin{array}}
\newcommand{\ea}{\end{array}}

\newcommand{\bbG} {\mathbb{G}}		
\newcommand{\bbZ} {\mathbb{Z}}		
\newcommand{\bbP} {\mathbb{P}}
\newcommand{\bbA} {\mathbb{A}}

\newcommand\INTO{\ar@{^{(}->}[r]}

\newcommand{\PGL}{\operatorname{PGL}}

\newcommand{\id}{\operatorname{id}}

\newcommand{\Sym}{\operatorname{\Sigma}}

\newcommand{\trdeg}{\operatorname{trdeg}}

\newcommand{\Spec}{\operatorname{Spec}}

\newcommand{\Char}{\operatorname{char}}

\newcommand{\presectionspace}{\vspace{0.2cm}} 

\begin{document}

\author{Tran-Trung Nghiem} 
\address{Département de Mathématiques et Applications\\Ecole Normale Supérieure\\Paris, France}
\email{tran-trung.nghiem@ens.fr}
\thanks{Tran-Trung Nghiem was supported by a Fondation Hadamard Scholarship.}

\author{Zinovy Reichstein}
\address{Department of Mathematics\\University of British Columbia\\ BC, Canada V6T 1Z2}
\email{reichst@math.ubc.ca}
\thanks{Zinovy Reichstein was partially supported by
National Sciences and Engineering Research Council of
Canada Discovery grant 253424-2017.}

\title[Fields of cross-ratios]{On a rationality problem 
 for fields of cross-ratios II}



\begin{abstract}
 Let $k$ be a field,  $x_1, \dots, x_n$ be independent variables and
$L_n = k(x_1, \dots, x_n)$. The symmetric group $\Sym_n$ acts on $L_n$ by permuting the variables, and
the projective linear group $\PGL_2$ acts by 
\[ \begin{pmatrix} a & b \\ c & d \end{pmatrix} \colon x_i \mapsto \frac{a x_i + b}{c x_i + d} \]
for each $i = 1, \ldots, n$. The fixed field $L_n^{\PGL_2}$ is called ``the field of cross-ratios".
Given a subgroup $S \subset \Sym_n$, H.~Tsunogai asked whether $L_n^S$ rational over $K_n^S$.  
When $n \geqslant 5$ the second author has shown that $L_n^S$ is rational over $K_n^S$ if and only if
$S$ has an orbit of odd order in $\{ 1, \dots, n \}$. In this paper we answer Tsunogai's question for $n \leqslant 4$.
\end{abstract}

\maketitle

\presectionspace
\section{Introduction}

Let $k$ be a base field, $n \geqslant 1$ be an integer, 
$x_1, \dots, x_n$ be independent variables, and
$L_n = k(x_1, \dots, x_n)$. The group $\PGL_2$ acts on $L_n$ 
via
\[ \begin{pmatrix} a & b \\ c & d \end{pmatrix} \cdot x_i \to 
\frac{a x_i + b}{c x_i + d} \]
for $i = 1, \dots, n$. The field of invariants 
$K_n = L_n^{\PGL_2}$ is 
generated over $k$ by the cross-ratios 
\begin{equation} \label{e.cr}
\frac{(x_i - x_1)(x_3 - x_2)}{(x_i - x_2)(x_3 - x_1)} 
 \end{equation}
for $i = 4, \ldots, n$. For this reason $K_n$ is often called the field of cross-ratios. (If $n \leqslant 3$, then $K_n = k$.) 
The natural action of the symmetric group $\Sym_n$ on $L_n$ 
by permuting the variables descends to a $\Sym_n$-action on $K_n$. Let $S$ be a subgroup of $\Sym_n$.
Motivated by the Noether problem, H. Tsunogai asked the following question~\cite[Introduction]{tsunogai}.

\begin{question} \label{q.main} 
Is $L_n^S$ is rational over $K_n^S$?
\end{question}

For $n \geqslant 5$ the second author answered Question 1 as follows; see~\cite{reichstein}.

\begin{theorem} \label{thm.old} 
Let $S$ be a subgroup of the symmetric group $\Sym_n$, where $n \geqslant 5$.
Then the following conditions are equivalent:

\begin{enumerate}
    \item[\rm{(a)}] $L_n^S$ is rational over $K_n^S$,
     \item[\rm{(a)}] $L_n^S$ is unirational over $K_n^S$,
     \item[\rm{(c)}] $S$ has an orbit of odd order in $\{ 1, \dots, n \}$.
\end{enumerate}
\end{theorem}

The purpose of this paper is to address Question~\ref{q.main} in the case where $n \leqslant 4$.
For $n \leqslant 3$, there is an easy answer. 
Here, as we mentioned above, $K_n = k$ and thus $K_n^S = k$ for any $S \subset \Sym_n$. 
In other words, for $n \leqslant 3$, Question~\ref{q.main} reduces to the following 
special case of the Noether Problem: Is $L_n^S$ rational over $k$? 
The answer is known to be ``yes" for every subgroup $S \subset \Sym_n$ ($n \leqslant 3$); see,~\cite[Theorem 3.3]{kang-wang}.

The case, where $n = 4$ is more delicate. Our main result is as follows.

\begin{theorem} \label{thm.main} Let $S$ be a subgroup of $\Sym_4$.

\smallskip
(a) Assume $S$ is not isomorphic to a cyclic group of order $4$.
Then $L_4^S$ is rational over $K_4^S$ for any base field $k$.

\smallskip
(b) Assume $S$ is cyclic of order $4$ and $\Char(k) \neq 2$. Then $L_4^S$ is rational over $K_4^S$ if and only if $k$ contains a primitive $4$th root of unity.

\smallskip
(c) Assume $S$ is cyclic of order $4$ and $\Char(k) = 2$. Then $L_4^S$ is rational over $K_4^S$.
\end{theorem}

Note that the symmetric group $\Sym_4$ has exactly $11$ subgroups up to conjugacy; see~\cite{gpw}. Part (a) covers $10$ of them. Note also that for $n \leqslant 3$ and for $n \geqslant 5$ the answer to Question~\ref{q.main} is independent of the base field $k$. For $n \leqslant 3$, it is always "yes", and for $n \geqslant 5$, it depends only on $n$ and the subgroup $S \subset \Sym_n$, up to conjugacy;
see Theorem~\ref{thm.old}. A cyclic group $S$ of order $4$ in $\Sym_4$ represents the only instant where the answer to Question~\ref{q.main} depends on $k$.

We will view $L_n$, $K_n$, $L_n^S$ and $K_n^S$ as the function fields of $(\bbP^1)^n$, $(\bbP^1)^n/\PGL_2$, $(\bbP^1)^n/S$ and $(\bbP^1)^n/(\PGL_2 \times S)$,
respectively. Here and in the sequel $X/G$ will denote the rational (or Rosenlicht) quotient variety for the action of an algebraic group $G$ 
on an algebraic variety $X$ defined over $k$. Recall that the Rosenlicht quotient $X/G$ is only defined up to birational equivalence and that $k(X/G)= k(X)^G$.
For details of this construction and further references, see~\cite[Section 2]{rs}. 

The remainder of this paper will be devoted to proving Theorem~\ref{thm.main}. 
Before proceeding with the proof we would like to explain a new phenomenon which arises 
in this case and which motivated our interest in Question~\ref{q.main} for $n = 4$.
When $n \geqslant 5$, the $\PGL_2$-action on $(\bbP^1)^n/S$ is generically free. (For the definition of a generically free action, see the beginning of Section~\ref{sect.prel}.) 
Consequently, the natural projection 
\[ \pi_S \colon (\bbP^1)^n/S \to (\bbP^1)^n/(\PGL_2 \times S) \]
is a $\PGL_2$-principal homogeneous space
over the generic point $\Spec(K_n^S)$ of $(\bbP^1)^n/(\PGL_2 \times S)$. It is shown in~\cite{reichstein} that $L_n^S$ is rational over $K_n^S$
if and only if this principal homogeneous space is split; the proof of Theorem~\ref{thm.old} in~\cite{reichstein} is based 
on this observation. When $n = 4$, $\pi_S$ is also a $\PGL_2$-homogeneous space over the generic point of $(\bbP^1)^n/(\PGL_2 \times S)$, but it may not be principal.
More precisely, the geometric fibers of $\pi_S$ in general position are isomorphic to $\PGL_2/S[4]$, where $S[4] = S \cap V[4]$ 
is the intersection of $S$ with the Klein $4$-subgroup
\begin{equation} \label{e.klein}  V[4] = \{ \id, (12)(34), (13)(24), (14)(23) \},  \end{equation}
suitably embedded in $\PGL_2$. 
It is easy to see that $S[4] = \{ 1 \}$ if and only if $S$ has a fixed point in $\{ 1, 2, 3, 4 \}$. In this case $\pi_S$ is again
a principal homogeneous space over $\Spec(K_4^S)$, and the same 
argument as in~\cite{reichstein} shows that $L_4^S$ is rational over $K_4^S$.
If $S[4] \neq 1$, then the arguments from~\cite{reichstein} no longer apply, and a different approach is required.

\section{First reductions}
\label{sect.prel}

Recall that the action of an algebraic group $G$ on an irreducible algebraic variety $X$ defined over a field $k$ is called generically free if
there exists a dense open $G$-invariant subvariety $X_0 \subset X$ such that the stabilizer $G_{x_0}$ is trivial for every $\overline{k}$-point $x_0 \in X_0$.
Here, as usual, $\overline{k}$ denotes the algebraic closure of $k$.

Let $B$ denote the Borel subgroup of upper triangular matrices in $\PGL_2$. Equivalently, $B \subset \PGL_2$ 
is the stabilizer of  the point $\infty = (1:0) \in \bbP^1$.
The following lemma is undoubtedly well known. For lack of a suitable reference, we include a short proof.

\begin{lemma} \label{lem.gen-free} Let $X_4 = (\bbP^1)^4/\Sym_4 \simeq \bbP^4$ be the space of
unordered $4$-tuples of points on $\bbP^1$. Then the $B$-action on $X_4$ is generically free.
\end{lemma}

\begin{proof} Denote the stabilizer of the unordered  $4$-tuple of points $\{ p_1, \ldots, p_4 \} \in (\bbP^1)^4 / \Sym_4$ by 
$H_{ p_1, \ldots, p_4 }$. Assume the contrary: $H_{p_1, \ldots, p_4} \neq 1$ for $p_1, \ldots, p_4$ in general position.

Now let $q_1, \ldots, q_5$ be a $5$-tuple of points in $\bbP^1$. Translating $q_5$ to $\infty$ by a suitable element of $\PGL_2$,
we obtain a $5$-tuple of the form $p_1, \ldots, p_4, \infty$. If $q_1, \ldots, q_5$ are in general position, then so are $p_1, \ldots, p_4$.
Hence, by our assumption, $H_{p_1, \ldots, p_4} \neq 1$. Since $H_{p_1, \ldots, p_4}$ is a subgroup of $B$, we see that $H_{p_1, \ldots, p_4}$
stabilizes $\infty$. Thus $H_{p_1, \ldots, p_4}$ stabilizes the unordered $5$-tuple $\{ p_1, \ldots, p_4, \infty \}$. 
Since the unordered $5$-tuple $\{ q_1, \ldots, q_5 \}$ lies in the same $\PGL_2$-orbit
as $\{ p_1, \ldots, p_4, \infty \}$, we conclude that the stabilizer of $\{q_1, \ldots, q_5 \}$ in $\PGL_2$ is non-trivial. 
On the other hand, it is well known that the stabilizer of an unordered $5$-tuple of points of $\bbP^1$ 
in general position is trivial, a contradiction.
\end{proof}

 In the sequel we will denote the field $L_4^B$ by $F_4$. The various invariant fields we are interested in are pictured in the diagram below. 
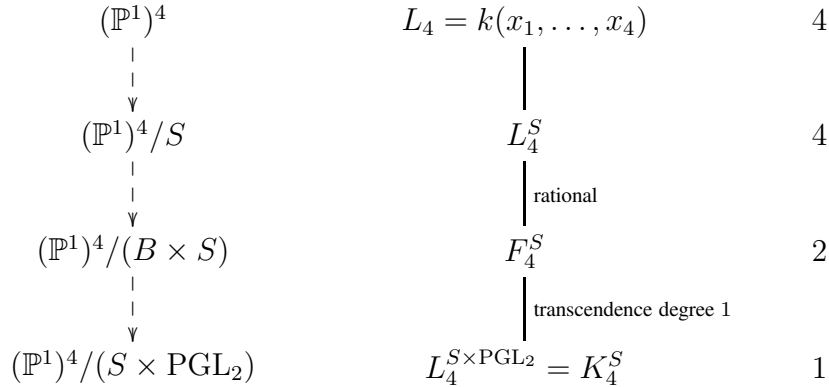
\[  \xymatrix{ 
  (\bbP^1)^4 \ar@{-->}[d]  \\ 
  (\bbP^1)^4/S   \ar@{-->}[d] \\ 
  (\bbP^1)^4/(B\times S) \ar@{-->}[d] \\                               
   (\bbP^1)^4/(S \times \PGL_2)} \quad \quad \quad \quad 
   \xymatrix{L_4 = k(x_1, \dots, x_4) \ar@{-}[d]  & & 4 \\
  L_4^{S}   \ar@{-}[d]^{\text{rational}}  & & 4 \\ 
  F_4^{S} \ar@{-}[d]^{\text{transcendence degree $1$}} & & 2    \\                           
 L_4^{S \times \PGL_2} = K_4^{S} & & 1} 
\]
Here the fields in the middle column represents the function fields of the varieties on the left. The right column lists the dimension of
each variety (or equivalently the transcendence degree of its function field) over $k$.
We now proceed with the main result of this section.

\begin{proposition} \label{prop.prel}
Let $S$ be a subgroup of $\Sym_4$.
Then the following are equivalent:

\smallskip
(a) $L_4^S$ is rational over $K_4^S$,

\smallskip
(b) $L_4^S$ is unirational over $K_4^S$,

\smallskip
(c) $F_4^S$ is unirational over $K_4^S$,

\smallskip
(d) $F_4^S$ is rational over $K_4^S$.
\end{proposition}

\begin{proof}[Proof of Proposition~\ref{prop.prel}]
The implication (a) $\Longrightarrow$ (b) $\Longrightarrow$ (c) are obvious.
The field extension $F_4^S/K_4^S$ is of transcendence degree $2 - 1 = 1$. Hence,
the implication (c) $\Longrightarrow$ (d) follows from L\"uroth's theorem~\footnote{Recall that L\"uroth's theorem asserts that
a unirational field extension of transcendence degree $1$ is rational; a proof can be found,
e.g., in~\cite[Chapter 8]{jacobson} or \cite[Chapter 10]{vdw}.}.

To prove the remaining implication (d) $\Longrightarrow$ (a), it suffices to show that $L_4^S$ is rational over $F_4^S$.
By Lemma~\ref{lem.gen-free} the action of $B$ on the space $(\bbP^1)^4/\Sym_4 \simeq \bbP^4$ 
of unordered $4$-tuples of points on $\bbP^1$ is generically free. Hence, so is the
action of $B$ on $(\bbP^1)^4/S$ for any subgroup $S \subset \Sym_4$. Since $B$ is a special group (see~\cite[Proposition II.1.2.1]{serre-gc}), 
this implies that $(\bbP^1)^4/S$ is birationally isomorphic to $\big( \, (\bbP^1)^4/(B \times S) \, \big) \times B$. Since $B$ is
a rational $2$-dimensional variety over $k$, we conclude that $L_4^S$ is rational of transcendence degree $2$ over $F_4^S$, as claimed.
\end{proof}

\section{Proof of Theorem~\ref{thm.main}(a)}

Let $S$ be a subgroup of $\Sym_4$ and consider the exact sequence
\begin{equation} \label{e.S_4} 
1 \to S[4] \to S \to S/S[4] \to 1 \, , 
\end{equation}
where
$V[4]$ is the Klein $4$-subgroup as in~\eqref{e.klein}, $S[4] := S \cap V[4]$, and $S/S[4]$ is a subgroup of 
$\Sym_4/V[4] \simeq \Sym_3$. Note that $S[4]$ acts trivially on $K_4$, and $S/S[4]$ acts faithfully.

\begin{lemma} \label{lem.split}
The sequence~\eqref{e.S_4} does not split if and only if $S \subset \Sym_4$ is a cyclic subgroup of order $4$. 
\end{lemma}

\begin{proof}[Proof of Lemma~\ref{lem.split}]
Observe that the sequence~\eqref{e.S_4} splits in the following two cases:
\[ \text{(i) if $S[4] = 1$} \quad \quad \text{or} \quad \quad \text{(ii) if $S = \Sym_4$.} \]

(i) is obvious, and (ii) follows from the fact that 
$\Sym_3$, naturally embedded into $\Sym_4$, is a complement to $V[4]$.

(ii) implies that the sequence~\eqref{e.S_4} splits whenever $S[4] = V[4]$. Thus we may assume without loss of generality
that $S[4]$ has order $2$. 
Now $S/S[4]$ is a subgroup of
$\Sym_3$, so has order $1$, $2$, $3$ or $6$. Let us consider these possibilities in turn.

If $|S/S[4]| = 1$ or $3$, then $|S| = 2$ or $6$, respectively. Clearly,~\eqref{e.S_4} splits in both cases. 

If $|S/S[4]| = 6$, then $|S| = 12$, so $S = A_4$ is the alternating group, and $S[4] = V[4]$, contradicting our assumption that $|S[4]| = 2$. 

This leaves us with the case, where $|S/S[4]| = 2$, i.e., $|S| = 4$. Up to conjugacy there are only two subgroups of order $4$ in $\Sym_4$, namely $V[4]$ and a cyclic subgroup generated by a $4$-cycle. 
Clearly $S \neq V[4]$, because we are assuming that $S[4] = S \cap V[4]$ has order $2$. Thus $S$ is the group of order $4$
generated by a $4$-cycle $\sigma$. In this case $S[4] = \langle \sigma^2 \rangle$, and the sequence~\eqref{e.S_4} does not split.
\end{proof}

We are now ready to complete the proof of Theorem~\ref{thm.main}(a).  

\smallskip
Case 1: $S[4] = 1$. This is equivalent to the condition that $S$ has a fixed point in $\{ 1, 2, 3, 4\}$.
As we mentioned in the Introduction, in this case the $\PGL_2 \times S$-action on $(\bbP^1)^4$ is generically free, and the same argument used 
to prove Theorem~\ref{thm.old} in \cite{reichstein} goes through unchanged. (For a self-contained proof, see Remark~\ref{rem.case1} below.) 
We conclude that $L_4^S$ is rational over $K_4^S$.

\smallskip
Case 2: $S \subset \Sym_4$ is not cyclic of order $4$. 
By Lemma~\ref{lem.split}, the sequence~\eqref{e.S_4} splits. Let $S' \subset S$ be a complement to $S[4]$.
By Case 1, $L_4^{S'}$ is rational over $K_4^S$. The diagram 
\[  \xymatrix{L_4^{S'}   \ar@{-}[d]  & \ar@{-}@/^1pc/[dd]^{\text{rational}}  \\
 L_4^{S} \ar@{-}[d] \ar@{-}@/^1pc/[d]^{\text{unirational}} &      \\                           
 K_4^{S}  = K_4^{S'} & }\]
now shows that $L_4^S$ is unirational over $K_4^S$. By Proposition~\ref{prop.prel}, we conclude that $L_4^S$ is rational over $K_4^S$.
\qed

\begin{remark} \label{rem.case1} 
For the sake of completeness we will now outline a self-contained geometric proof of Theorem~\ref{thm.main}(a) in Case 1.

As we mentioned above, since $S[4] = 1$, $S$ fixes an element of $\{1, 2, 3, 4\}$. 
After replacing $S$ by a conjugate in $\Sym_4$ we may assume without loss of generality that $S$ fixes $4$, i.e., $S \subset \Sym_3$. Consider the diagram of
$\PGL_n$-equivariant rational maps
\[  \xymatrix{ 
  (\bbP^1)^4 \ar@{-->}[d]  \ar@{->}[dr]^{p_4}    &        &  \\ 
  (\bbP^1)^4/S   \ar@{-->}[d]^{\pi}    \ar@{-->}[r]^{\overline{p_4}}             & \bbP^1  \ar@{<->}[r]^{\sim \quad \; \; } & \PGL_2/B \\                               
   (\bbP^1)^4/(S \times \PGL_2). & & }
\]
Here $B$ is the Borel subgroup of upper-triangular matrices in $\PGL_2$, as in Section~\ref{sect.prel}, and
$p_4 \colon (\bbP^1)^4 \to \bbP^1$ is the projection to the $4$th factor. Since $S \subset \Sym_3$, $p_4$ descends to a $\PGL_2$-equivariant map
$(\bbP^1)^4/S \dasharrow \bbP^1$, which we denote by $\overline{p_4}$.  The function fields of $(\bbP^1)^4$, $(\bbP^1)^4/S$ and $(\bbP^1)^4/(S \times \PGL_2)$
are $L_4$, $L_4^S$ and $K_4^S$, respectively; see the diagram in Section~\ref{sect.prel}.
The $\PGL_2$-action on $(\bbP^1)^4/S$ is generically free; hence, $\pi$ is a $\PGL_2$-torsor over the generic point
of $(\bbP^1)^4/(S \times \PGL_2)$. Our goal is to prove that this torsor is split; this will show that $(\bbP^1)^4/S$ is birational to
$(\bbP^1)^4/(S \times \PGL_2) \times B$ and thus $L_4^S$ is rational over $K_4^S$.

Since there exists a $\PGL_2$-equivariant map from the total space of $\pi$ to $\PGL_2/B$, $\pi$ admits reduction of structure to
the Borel subgroup $B$. In other words, the class $[\pi]$ of $\pi$ in $H^1(K_4^S, \PGL_2)$ lies in the image of the natural map $H^1(K_4^S, B) \to H^1(K_4^S, \PGL_2)$. As we mentioned in the proof of Proposition~\ref{prop.prel}, $B$ is a special group. In particular $H^1(K_4^S, B) = 1$, 
and the desired conclusion follows.  
\end{remark}

\section{A preliminary computation}

The approach we used in the previous section does not work when $S$ is a cyclic group of order $4$. To prove parts (b) and (c) of Theorem~\ref{thm.main}
we will resort to explicit calculations in the next section. The following lemma will facilitate these calculations.

\begin{lemma} \label{lem.no-name} Let $k$ be a field and $a$, $u$ be independent variables over $k$.

\smallskip
(a) Let $\sigma$ be an automorphism of $k(a, u)/k$ of order $2$
given by $\sigma(a) = 1-a$ and $\sigma(u) = - \dfrac{1}{u}$. Assume $\Char(k) \neq 2$.
Then $k(a, u)^{\langle \sigma \rangle}$ is rational over $k(a)^{\langle \sigma \rangle}$ if and only if $k$ contains a primitive $4$th root of unity.

\smallskip
(b) Assume $\Char(k) = 2$. Let $\sigma$ be an automorphism of $k(a, u)/k$ of order $2$
given by $\sigma(a) = a + 1$ and $\sigma(u) = u + 1$. 
Then $k(a, u)^{\langle \sigma \rangle}$ is rational over $k(a)^{\langle \sigma \rangle}$.
\end{lemma}

\begin{proof} (a) Set $b = 1 - 2a$, so that $\sigma(b) = - 1 + 2a = - b$. Since $\Char(k) \neq 2$, we have $k(a,u) = k(b,u)$ and $k(a) = k(b)$. 
Thus part (a) is equivalent to

\smallskip
(a$'$) \textit{$k(b,u)^{\langle \sigma \rangle}$ is rational over $k(b)^{\langle \sigma \rangle}$ if and only if $k$ contains a 4th root of unity. }

\smallskip
Clearly $k(b)^{\langle \sigma \rangle} = k(x)$, where $x = b^2$. We claim that $k(b, x)^{\langle \sigma \rangle} = k(x, y, z)$, 
where \[y = \frac{b}{2} (u + \frac{1}{u}) \quad \text{and} \quad z = \frac{1}{2} (u  - \frac{1}{u}). \]
Indeed, one readily checks that $x, y, z \in k(b, x)^{\langle \sigma \rangle}$. 
It remains to show that $[k(b, u) : k(x, y, z)] \leqslant 2$; the claim will immediately follow from the diagram below.  
\[ \xymatrix{k(b, u)  \ar@{-}[d] \ar@{-}@/^1pc/[d]^{\text{degree $2$}} & \ar@{-}@/^1pc/[dd]^{\text{degree $\leqslant 2$}}  \\
k(b, u)^{\langle \sigma \rangle} \ar@{-}[d]  &   \\
k(x, y, z). & } 
 \]
To show that $[k(b, u) : k(x, y, z)] \leqslant 2$, note that $k(b,u)= k(x,y,z)(u)$ and $u$ satisfies the quadratic equation 
$u^2 - 2zu - 1 = 0$ over $k(x,y,z)$. This proves the claim. 

We have thus reduced part (a) to the following assertion:

\smallskip
(a$''$) \textit{$k(x, y, z)$ is rational over $k(x)$ if and only if $k$ contains a primitive $4$th root of unity. }

\smallskip
Note that $k(x, y, z) = k(b, u)^{\langle \sigma \rangle}$ is of transcendence degree $2$ over $k$ and 
\[
y^2 - xz^2 - x = 0.  
\]
This equation defines a conic in $\bbA^2$ over the field $k(x)$. Since this conic is absolutely irreducible, 
$k(x, y, z)$ is the function field of this conic. This field is rational over $k(x)$ if and only if the
projective conic
\[ \label{proj.conic.charnot2}
Y^2 - x Z^2 - x W^2 = 0 
\]
has a $k(x)$-point. Here $Y$, $Z$ and $W$ are homogeneous coordinates in $\bbP^2$. It thus remains to show that

\smallskip
(a$'''$) \textit{The quadratic form $q(Y, Z, W) = Y^2 - x Z^2 - x W^2$ is isotropic over $k(x)$ if and only if $k$ contains a primitive $4$th root of unity. }

\smallskip
Suppose $k$ contains a primitive $4$th root of unity. Denote it by $i$. Then $q (0 , i , 1) = 0$, so $q$ is isotropic over $k(x)$. 

Conversely, assume $q$ is isotropic over $k(x)$. That is, $q(A(x), B(x), C(x)) = 0$ for some
$A(x), B(x), C(x) \in k(x)$, not all zero. After clearing denominators, we may assume that $A(x)$, $B(x)$ and $C(x)$ are polynomials with coefficients in
$k$, and $(A(0), B(0), C(0)) \neq (0, 0, 0)$. Substituting $x = 0$ into the equation
$A(x)^2 = x(B(x)^2 + C(x)^2)$, we see that $A(0) = 0$. Thus the left hand side is divisible by $x^2$ and consequently, $B(0)^2 + C(0)^2 = 0$, where
$B(0)$ and $C(0)$ are not both $0$. This means that neither can be $0$, and $B(0)/C(0)$ is a primitive $4$th root of unity in $k$. This completes the proof of (a$'''$) and thus of part (a).

 \smallskip
(b) Set $ x = a(1+a), y = u(1 + u), z = a + u$. Note that $x$, $y$, $z$ are invariant under $\sigma$ and $[k(a, u): k(x, y, z)] \leqslant 2$, because
$k(a, u)$ is generated by $u$ over $k(x, y, z)$, and $u$ satisfies the quadratic equation $u^2 + u + y = 0$ over $k(x, y, z)$. Using the same argument 
as in part (a), we see that
$k(a)^{\langle \sigma \rangle} = k(x)$, $k(a,u)^{\langle \sigma \rangle} = k(x,y,z)$ and $k(x, y, z)$ is the function field of the affine quadric 
\[ z^2 + z + y + x = 0 \]
over $k(x)$. Equivalently, $k(x, y, z)$ is the function field of the projective conic
\[ Z^2 + ZW + YW + xW^2 = 0, \]
over $k(x)$, where $Y$, $Z$ and $W$ are homogeneous coordinates in $\bbP^2$. This conic has a $k(x)$-point $(Y: Z: W) = (x: 1 : 1)$.
Thus $k(a,u)^{\langle \sigma \rangle} = k(x,y,z)$ is rational over $k(a)^{\langle \sigma \rangle} = k(x)$. 
\end{proof}

\begin{remark} \label{rem.referee}
In this remark we will briefly outline an alternative proof of Lemma~\ref{lem.no-name} 
which was suggested to us by the referee.

(a) The argument on p.~371 in~\cite{hkk} shows that $k(a, u)^{\langle \sigma \rangle}$ is the function field of the Brauer-Severi variety of the quaternion algebra
$(x, -1)$ over the field $k(a)^{\langle \sigma \rangle} = k(x)$, where $x = b^2$, as in our proof. Thus $k(a, u)^{\langle \sigma \rangle}$ is rational over $k(a)^{\langle \sigma \rangle} = k(x)$ if and
only if the quaternion algebra $(x, -1)$ splits over $k(x)$. On the other hand, $(x, -1)$ splits over $k(x)$ if and only if $-1$ is a square in $k$; 
see, e.g.,~\cite[Chapter 1, Example 1.3.8]{gille_szamuely}. 

(b) follows from~\cite[Theorem 2.2]{kang-wang}. \qed
\end{remark}

\section{Conclusion of the proof of Theorem~\ref{thm.main}}

Our goal is to prove parts (b) and (c) of Theorem~\ref{thm.main}. We may assume without loss of generality 
that $S = \langle \sigma \rangle \subset \Sym_4$, where $\sigma$ is the $4$-cycle $(1 \, 2 \, 3 \, 4 )$.

We begin by constructing a convenient birational model for $(\bbP^1)^4/B$, where the action of $\Sym_4$ is particularly transparent.
(Recall that a priori the rational quotient $(\bbP^1)^4/B$ is only defined up to birational isomorphism.)
Let $V$ be the $4$-dimensional $k$-vector space, $x_1, \dots, x_4$ be a basis for the dual space $V^*$, 
$V_1$ is the 1-dimensional subspace of $V$ spanned by the vector $(1, 1, 1, 1)$ and $\overline{V} = V/ V_1$.
The dual space $\overline{V}^{\, *}$ is the $3$-dimensional subspace of $V^*$ consisting of linear functions $\lambda_1 x_1 + \ldots + \lambda_1 x_4$ such that $\lambda_1 + \ldots + \lambda_4 = 0$.

The Borel subgroup $B$ of upper triangular matrices in $\PGL_2$ decomposes as a semidirect product $U \rtimes T$, where $T \simeq \bbG_m$ is the diagonal maximal
torus and $U \simeq \bbG_a$ is the group of matrices of the form $\begin{pmatrix} 1 & \ast \\ 0 & 1 \end{pmatrix}$.
One readily checks that the rational quotient $(\bbP^1)^4/U$ is $\Sym_4$-equivariantly birationally isomorphic to the two-dimensional affine space
$\bbA(\overline{V})$ and the rational quotient 
$(\bbP^1)^4/B$ is $\Sym_4$-equivariantly birationally isomorphic to the $2$-dimensional projective space \[ ((\bbP^1)^4/U)/T \simeq \bbA(\overline{V})/\bbG_m \simeq \bbP(\overline{V}). \]
In other words, $F_4 \stackrel{\rm def}{=} L_4^B$ is $\Sym_4$-equivariantly isomorphic to the function field of $\bbP(\overline{V})$ over $k$.
That is, 
\[ F_4 = k \big(\; \frac{\lambda_1 x_1 + \ldots + \lambda_4 x_4}{\mu_1 x_1 +  \ldots + \mu_4 x_4} \; | \; \lambda_1 + \ldots \lambda_4 = \mu_1 + \ldots + \mu_4 = 0 \big)\, ; \]
cf.~formula (4) on p. 904 in~\cite{tsunogai}. 

\begin{proof}[Proof of Theorem~\ref{thm.main}(b)]
Assume $\Char(k) \neq 2$. Then the linear functions 
\[ \begin{matrix}  w = - x_1 - x_2 + x_3 + x_4, \\
 y = - x_1 +  x_2 + x_3 - x_4, \\
 z = -x_1 + x_2 - x_3 + x_4 \\
 \end{matrix} \]
form a $k$-basis for $\overline{V}^{\, *}$. 
Note that 
\begin{align*} x_4 - x_1 = \frac{1}{2}(w + z), &   &
 x_3 - x_2 = \frac{1}{2}(w - z), \\
 x_4 - x_2 = \frac{1}{2}(w - y), &  &
 x_3 - x_1 = \frac{1}{2}(w + y).
\end{align*}
Now recall that by~\eqref{e.cr},
\begin{equation} \label{e.cr2} \text{$K_4 = k(a)$, where $a = \dfrac{(x_4 - x_1)(x_3 - x_2)}{(x_4 - x_2)(x_3 - x_1)}$.} 
\end{equation}
The generic fiber of the natural projection map $(\bbP^1)^4/B \to (\bbP^1)^4/\PGL_2$ 
is the quadric in $\bbP_{K_4}^3$ given by
\[ (x_4 - x_1) (x_3 - x_2) = a(x_4 - x_2)(x_3 - x_1) \]
or equivalently, $w^2 - z^2 = a(w^2 - y^2)$ or
\[ (1 - a) w^2 - z^2 + a y^2 = 0 \, . \]
Setting $u = w/y$ and $t = z/y$, we see that $F_4 = k(a, u, t)$, where 
\begin{equation} \label{e.affine-conic} (1-a) u^2 - t^2 + a = 0. 
\end{equation}
We claim that $F_4^{\langle \sigma^2 \rangle} = k(a, u)$. Indeed, since 
$\sigma(w) = - y$, $\sigma(y) = w$, $\sigma(z) = - z$, and $\sigma(a) = 1 - a$, we see that
\[ \sigma^2(a) = a, \; \; \sigma^2(u) = u, \; \; \; \text{and} \; \; \; \sigma^2(t) = - t. \]
Now consider the following diagram of field extensions
\begin{equation}\label{fdia} \xymatrix{F_4  \ar@{-}[d] \ar@{-}@/^1pc/[d]^{\text{degree $2$}} & \ar@{-}@/^1pc/[dd]^{\text{degree $\leqslant 2$}}  \\
F_4^{\langle \sigma^2 \rangle} \ar@{-}[d]  &   \\
k(a, u) \ar@{-}[d]  &      \\                           
 K_4 = k(a). & } 
 \end{equation}
Here $[F_4: F_4^{\langle \sigma^2 \rangle}] = 2$ because $\sigma^2$ is an automorphism of order $2$, and $[F_4 : k(a, u)] \leqslant 2$ because $t$ satisfies a quadratic equation over $k(u, a)$; see~\eqref{e.affine-conic}. We conclude that $F_4^{\langle \sigma^2 \rangle} = k(a,u)$. This proves the claim.

Note that since $\trdeg_k(F_4^{\langle \sigma^2 \rangle}) = \trdeg_k(F_4) = 2$, $a$ and $u$ are algebraically independent over
$k$. The group $S/\langle \sigma^2 \rangle \simeq \bbZ/2\bbZ $ acts on $F_4^{\langle \sigma^2 \rangle} = k(a,u)$ by $\sigma \colon a \mapsto 1 - a$ and $\sigma \colon u \mapsto -1/u$. Lemma~\ref{lem.no-name}(a) now tells us that
$F_4^{S} = (F_4^{\langle \sigma^2 \rangle})^{S/{\langle \sigma^2 \rangle}} = k(a, u)^{\langle \sigma \rangle}$ 
is rational over $K_4^S = k(a)^{\langle \sigma \rangle}$ if and only if $k$ contains a primitive $4$th root of unity.
\end{proof}

\begin{proof}[Proof of Theorem~\ref{thm.main}(c)] Now assume that $\Char k = 2$.
The linear functions 
\[ \begin{array}{l} w = x_1 + x_2 + x_3 + x_4, \\
y = x_1 + x_3 , \\
z = x_1 + x_4
\end{array} \]
form a $k$-basis for $\overline{V}^{\, *}$. Again by ~\eqref{e.cr2}, the fiber of the natural projection map $(\bbP^1)^4/B \to (\bbP^1)^4/\PGL_2$ 
over the generic point $\Spec(K_4) \to (\bbP^1)^4/\PGL_2$
is the quadric in $\bbP_{K_4}^3$ given by
\[(x_4 + x_1)(x_3 + x_2) = a(x_4 + x_2)(x_3 + x_1). \]
(Note that in characteristic $2$, $x_i - x_j$ is the same as $x_i + x_j$.)
In $w, y, z$-coordinates, this equation can be rewritten as $z(w + z) = a(w+y)y$, or equivalently, as 
\[ ay^2 + ayw + z^2 + zw = 0. \]
Setting $u = y/w$ , $t = z/w$, we see that $F_4 = k(a,u,t)$, where 
\begin{equation} \label{e.affine-conic2} au^2 + au + t^2 + t = 0.
\end{equation}
The action of $\sigma$ is given by $\sigma(w) = w$, $\sigma(y) = w + y$, $\sigma(z) = w + y + z$. Thus 
\[\sigma^2(w) = w,\; \; \; \sigma^2(y) = y, \; \; \; \sigma^2(z) = w + z \]
Examining the diagram of field extensions~\eqref{fdia} and arguing as in the proof of Theorem~\ref{thm.main}(b) above, we conclude that 
$F_4^{\langle \sigma^2 \rangle} = k(a, u)$, where $a$ and $u$ are algebraically independent over $k$. Once again,
the group $S/\langle \sigma^2 \rangle \simeq \bbZ/ 2 \bbZ$ acts faithfully on $F_4^{\langle \sigma^2 \rangle} = k(a,u)$ by $\sigma \colon a \mapsto 1 - a = a + 1$
and $\sigma \colon u \mapsto u + 1$. By Lemma~\ref{lem.no-name}(b), $F_4^S = k(a,u)^{\langle \sigma \rangle}$ is rational over $K_4^S = k(a)^{\langle \sigma \rangle}$, as desired.
\end{proof}

\section*{Acknowledgment} This paper is based on a summer research project conducted by the first author at the University of British Columbia\footnote{Due 
to the travel restrictions associated with the Cover-19 pandemic, the work was conducted virtually.}. We are grateful to the anonymous referee for contributing Remark~\ref{rem.referee} and for other constructive comments.

\end{document}